\documentclass{proc-l}
\usepackage{amssymb,amsthm,amsmath}
\newcommand{\ba}{\mathfrak{a}}
\newcommand{\bb}{\mathfrak{b}}
\newcommand{\bA}{\mathbb{A}}
\newcommand{\bQ}{\mathbb{Q}}
\newcommand{\bR}{\mathbb{R}}
\newcommand{\mJ}{\mathcal{J}}
\DeclareMathOperator{\Spec}{{Spec}}
\DeclareMathOperator{\Hom}{Hom}
\DeclareMathOperator{\Div}{{div}}

\newtheorem{theorem}{Theorem}[section]
\newtheorem{lemma}[theorem]{Lemma}
\theoremstyle{definition}
\newtheorem{definition}[theorem]{Definition}
\newtheorem{example}[theorem]{Example}
\theoremstyle{remark}
\newtheorem{remark}[theorem]{Remark}

\input xy
\xyoption{all}

\begin{document}

\title{A note on discreteness of $F$-jumping numbers}
\author{Karl Schwede}
\begin{abstract}
Suppose that $R$ is a ring essentially of finite type over a perfect field of characteristic $p > 0$ and that $\ba \subseteq R$ is an ideal.  We prove that the set of $F$-jumping numbers of $\tau_b(R; \ba^t)$ has no limit points under the assumption that $R$ is normal and $\bQ$-Gorenstein -- we do \emph{not} assume that the $\bQ$-Gorenstein index is not divisible by $p$.  Furthermore, we also show that the $F$-jumping numbers of $\tau_b(R; \Delta, \ba^t)$ are discrete under the more general assumption that $K_R + \Delta$ is $\bR$-Cartier.
\end{abstract}
\address{Department of Mathematics\\ University of Utah \\ Salt Lake City, UT, 84112}
\email{kschwede@umich.edu}
\thanks{The author was partially supported by a National Science Foundation postdoctoral fellowship and by NSF grant DMS-1064485/0969145.}

\maketitle

\section{Introduction}

The test ideal is an important and subtle object associated to ideals $\ba$ in positive characteristic rings $R$.   It measures the singularities of both the ambient ring and the elements of the ideal; see \cite{HaraYoshidaGeneralizationOfTightClosure}.  While the test ideal was initially introduced in the celebrated theory of tight closure of Hochster and Huneke (see \cite{HochsterHunekeTC1}), more recent interest in the test ideal has been in regards to its connection with the multiplier ideal -- a fundamental invariant of higher dimensional algebraic geometry in characteristic zero; see for example \cite{TakagiFormulasForMultiplierIdeals} or \cite{MustataYoshidaTestIdealVsMultiplierIdeals}.

Given a normal ring $R$ essentially of finite type over a perfect field of characteristic $p > 0$, an ideal $\ba \subseteq R$ and a real number $t \geq 0$, one can form the (big) test ideal $\tau_b(R; \ba^t)$ -- an object which measures both algebraic and arithmetic properties of $R$ and $\ba$.  Inspired by the test ideal's close relation with the multiplier ideal $\mJ(R, \ba^t)$, people have studied the numbers $t_i$ where $\tau_b(R; \ba^{t_i})$ changes.  That is, people have studied the \emph{$F$-jumping numbers}, see \cite{MustataTakagiWatanabeFThresholdsAndBernsteinSato}; real numbers which are by definition the $t_i > 0$ such that for every $\varepsilon > 0$,
\[
\tau_b(R; \ba^{t_i - \varepsilon}) \neq \tau_b(R; \ba^{t_i}).
\]

One easy to observe fact about multiplier ideals is that their jumping numbers are discrete and rational, at least when $R$ is $\bQ$-Gorenstein and normal; see \cite{EinLazSmithVarJumpingCoeffs}.  Here, by discrete we mean that the set of jumping numbers with respect to a fixed ideal have no limit points.
Because of this, various groups have recently worked to show that the $F$-jumping numbers of the test ideal are also discrete and rational; see \cite{HaraMonskyFPureThresholdsAndFJumpingExponents}, \cite{BlickleMustataSmithDiscretenessAndRationalityOfFThresholds}, \cite{BlickleMustataSmithFThresholdsOfHypersurfaces}, \cite{KatzmanLyubeznikZhangOnDiscretenessAndRationality}, and \cite{BlickleSchwedeTakagiZhang}.  In the most recent mentioned work, the author along with M. Blickle, S. Takagi, and W. Zhang, showed that the $F$-jumping numbers of test ideals formed a discrete set of rational numbers when $R$ is normal and $\bQ$-Gorenstein \emph{with index not divisible by $p > 0$.}  Recall that the \emph{index} of a $\bQ$-Gorenstein ring $R$ is the smallest natural number $n$ where $\omega_R^{(n)} = \O_{\Spec R}(nK_R)$ is locally free.

The most fundamental case left open is the case when $R$ is $\bQ$-Gorenstein, but of arbitrary index, see \cite[Question 6.1]{BlickleSchwedeTakagiZhang}.  We answer this question at least for discreteness.
\vskip 6pt
\hskip -12pt
{\bf Theorem \ref{MainTheorem}.} {\it
Suppose that $R$ is a normal domain essentially of finite type over an $F$-finite field.  Further suppose that $\ba \subseteq R$ is an ideal and $\Delta$ is an $\bR$-divisor on $X = \Spec R$ such that $K_X + \Delta$ is $\bR$-Cartier (for example, this holds if $\Delta = 0$ and $R$ is $\bQ$-Gorenstein).  Then, as $t$ varies, the $F$-jumping numbers of $\tau_b(R; \Delta, \ba^t)$ have no limit points -- they are discrete.
}
\vskip 6pt
\hskip -12pt
We also point out why the existing proofs of \emph{rationality} do not seem to work in the case that $R$ is $\bQ$-Gorenstein with index divisible by $p$.

Recently, in \cite{DeFernexHaconSingsOnNormal}, de Fernex and Hacon gave a definition of the multiplier ideal without the $\bQ$-Gorenstein assumption and asked the question of whether discreteness and rationality of the $F$-jumping numbers still holds in this context.  Following this, Urbinati showed that rationality need not hold but gave some evidence that discreteness may hold in general, see \cite{UrbinatiDiscOfNonQGorenVars}.  This suggests that one should not expect rationality to hold in positive characteristic either.
\vskip 3pt
\hskip -12pt{\it Acknowledgements:}  The author would like to thank Manuel Blickle and Kevin Tucker for several stimulating and valuable conversations.  The author would also like to thank the referee for numerous helpful suggestions.

\section{Definition of the test ideal}

We only give a very brief description of the big test ideal in this paper.  Please see \cite{BlickleSchwedeTakagiZhang} for a more detailed description of the test ideal.

First we fix some notation.  Given a ring $R$ of characteristic $p > 0$ and $M$ an $R$-module, we set $F^e_* M$ to be the $R$-module which agrees with $M$ as an additive group but where the $R$-module structure is defined by the rule $r . m = r^{p^e} m$.  Also recall that $R$ is said to be \emph{$F$-finite} if $F^e_* R$ is a finitely generated $R$-module.
\vskip 3pt
\noindent
  \begin{tabular}[c]{ll}
    {\scshape Convention:} & \parbox[t]{3.5in}{Throughout this note, all rings will be
  assumed to be $F$-finite.}\\
  \end{tabular}
\vskip 6pt
Recall that an $\bR$-divisor on a normal scheme $X$ is a formal linear combination of prime Weil divisors $D_i$ with real coefficients.  An $\bR$-divisor $D$ is called \emph{$\bR$-Cartier} if it is equal to an $\bR$-linear combination of Cartier divisors.

We now define the test ideal $\tau_b(R; \Delta, \ba^t \bb_1^{s_1} \dots \bb_m^{s_m})$.  We work in this greater generality because when proving our main theorem, we perturb our initial triple $(R, \Delta, \ba^t)$ to a new triple $(R, \Delta', \ba^t \bb_1^{s_1} \dots \bb_m^{s_m})$ which has the same test ideal.

\begin{definition} \cite{HochsterHunekeTC1}, \cite{HochsterFoundations}, \cite{SchwedeCentersOfFPurity}
Suppose that $R$ is an $F$-finite normal domain, $\Delta \geq 0$ is an $\bR$-divisor on $X = \Spec R$, $\ba, \bb_1, \dots \bb_m \subseteq R$ are non-zero ideals and $t, s_1, \dots, s_m \geq 0$ are real numbers.  Then the \emph{big test ideal} $\tau_b(R; \Delta, \ba^t \bb_1^{s_1} \dots \bb_m^{s_m})$ is defined to be the unique smallest non-zero ideal $J \subseteq R$ such that
\begin{equation}
\label{eqnDefnOfTestIdeal}
\phi\left(F^e_* (\ba^{\lceil t(p^e - 1) \rceil} \bb_1^{\lceil s_1(p^e - 1) \rceil} \dots \bb_m^{\lceil s_m(p^e - 1) \rceil} J) \right) \subseteq J
\end{equation}
for every $e \geq 0$ and every $\phi \in \Hom_R(F^e_* R(\lceil (p^e -1)\Delta \rceil), R)$. This ideal always exists in the context described.
\end{definition}

\begin{remark}
 In the case that $K_X + \Delta$ is $\bQ$-Cartier, the \emph{big test ideal} is known to equal the \emph{(finitistic) test ideal} (which we will not define here); see \cite{TakagiInterpretationOfMultiplierIdeals} and \cite{BlickleSchwedeTakagiZhang} for details.
\end{remark}

If all $\bb_i = R$, then we denote the associated big test ideal by $\tau_b(R; \Delta, \ba^t)$.  Likewise if $\Delta = 0$, then we denote the associated big test ideal using the notation $\tau_b(R; \ba^t \bb_1^{s_1} \dots \bb_m^{s_m})$.  Finally if the $\bb_i = (f_i)$ are principal, we denote the associated big test ideal by $\tau_b(R; \Delta, \ba^t f_1^{s_1} \dots f_m^{s_m})$

\begin{remark}\label{remTestIdealViaTestElements}
Given a non-zero element $c \in \tau_b(R; \Delta, \ba^t \bb_1^{s_1} \dots \bb_m^{s_m})$ (such an element is called a \emph{big sharp test element}), we note that
\begin{equation}\label{eqnTestIdealsViaTestElements}
\tau_b(R; \Delta, \ba^t \bb_1^{s_1} \dots \bb_m^{s_m}) = \sum_{e \geq 0} \sum_{\phi} \phi\left(F^e_* (c \ba^{\lceil t(p^e - 1) \rceil} \bb_1^{\lceil s_1(p^e - 1) \rceil} \dots \bb_m^{\lceil s_m(p^e - 1) \rceil}) \right)
\end{equation}
where the inner sum is over $\phi \in \Hom_R(F^e_* R(\lceil (p^e -1)\Delta \rceil), R)$.  To see this, simply note that the right side satisfies condition \ref{eqnDefnOfTestIdeal}, and it is by definition the smallest ideal containing $c$ satisfying condition \ref{eqnDefnOfTestIdeal}, note $\ba^0 = \bb_i^0 = R$.
\end{remark}

Suppose that $X = \Spec R$ is normal.  Then given $\phi \in \Hom_R(F^e_* R, R) \cong F^e_* \O_X((1-p^e)K_X)$, we may view $\phi$ as determining an effective Weil divisor linearly equivalent to $(1-p^e)K_X$.

\begin{definition}
We use $D_{\phi}$ to denote the Weil divisor associated to $\phi$ in this way.
\end{definition}

Given an $\bR$-divisor $\Delta \geq 0$ on $X$, one has an inclusion \begin{equation}\label{eqnNiceInclusion}\Hom_R(F^e_* R(\lceil (p^e - 1)\Delta \rceil), R) \subseteq \Hom_R(F^e_* R, R).\end{equation}  The following lemma gives a nice interpretation of this submodule.

\begin{lemma}
An element $\phi \in \Hom_R(F^e_* R, R)$ is contained inside the submodule $\Hom_R(F^e_* R(\lceil (p^e - 1)\Delta \rceil), R)$ if and only if $D_{\phi} \geq (p^e - 1)\Delta$.
\end{lemma}
\begin{proof}
Because all the module are reflexive the statement can be reduced to the case when $R$ is a discrete valuation ring and $\Delta = s \Div(x)$ where $x$ is the parameter for the DVR $R$ and $s \geq 0$ is a real number.  In this case, the inclusion from equation \ref{eqnNiceInclusion} can be identified with the  multiplication map $R \to R$ which sends $1$ to $x^{\lceil s(p^e - 1) \rceil}$.  Thus, $\phi \in \Hom_R(F^e_* R, R) \cong R$ is contained inside $\Hom_R(F^e_* R(\lceil (p^e - 1)\Delta \rceil), R) \cong x^{\lceil s(p^e - 1) \rceil} R$ if and only if $D_{\phi} \geq \lceil s(p^e - 1) \rceil \Div(x) = \lceil (p^e - 1) \Delta \rceil$.  However, since $D_{\phi}$ is integral, it is harmless to remove the round-up $\lceil \cdot \rceil$.
\end{proof}

\section{Discreteness of $F$-jumping numbers}

In this section we prove our main result.  We accomplish this by perturbing our triples $(R, \Delta, \ba^t)$ in order to reduce the discreteness statement to the case where the (log) $\bQ$-Gorenstein index is not divisible by $p > 0$.  First we need a lemma.

\begin{lemma}
\label{lemma:TwistIn}
 Suppose that $(X = \Spec R, \Delta, \ba^t \bb_1^{s_1} \dots \bb_m^{s_m})$ is a triple and that $\Delta = \Gamma + b \Div(f)$ for some $f \in R \setminus \{ 0 \}$ and nonnegative number $b \in \bR$.  Then
\[
 \tau_b(R; \Delta, \ba^t \bb_1^{s_1} \dots \bb_m^{s_m}) = \tau_b(R; \Gamma, f^b \ba^t  \bb_1^{s_1} \dots \bb_m^{s_m}).
\]
\end{lemma}
This type of statement is essentially obvious for multiplier ideals, but because of certain issues surrounding the construction of test ideals we have thus-far presented, it is somewhat less obvious in this context.  However, it is still quite straightforward especially from definition of the generalized test ideal by Hara-Yoshida-Takagi (the proof in that case uses the theory tight closure), see \cite{HaraYoshidaGeneralizationOfTightClosure} and \cite{TakagiInterpretationOfMultiplierIdeals}.  We provide a short proof here certainly acknowledging that this statement is obvious to experts.
\begin{proof}
 Choose $c$ to be a big sharp test element for both $(X, \Delta, \ba^t \bb_1^{s_1} \dots \bb_m^{s_m})$ and $(X, \Gamma, f^b \ba^t \bb_1^{s_1} \dots \bb_m^{s_m})$.  Then we know the test ideal $\tau_b(R; \Gamma, f^b \ba^t \bb_1^{s_1} \dots \bb_m^{s_m})$ equals,
\[
\begin{split}
 = \sum_{\phi,\,  D_{\phi} \geq (p^e - 1) \Gamma} \phi\left(F^e_* c f^{\lceil b(p^e - 1) \rceil} \ba^{\lceil t(p^e - 1) \rceil} \bb_1^{\lceil s_1(p^e - 1) \rceil} \dots \bb_m^{\lceil s_m(p^e - 1) \rceil} \right)\\
= \sum_{\phi,\, D_{\phi} \geq (p^e - 1)\Gamma + \Div{f^{\lceil b(p^e - 1) \rceil}}} \phi\left(F^e_* c  \ba^{\lceil t(p^e - 1) \rceil} \bb_1^{\lceil s_1(p^e - 1) \rceil} \dots \bb_m^{\lceil s_m(p^e - 1) \rceil} \right)\\
\subseteq \sum_{\phi,\,  D_{\phi} \geq (p^e - 1)\Gamma + \lceil b(p^e -1) \Div(f) \rceil} \phi\left(F^e_* c  \ba^{\lceil t(p^e - 1) \rceil} \bb_1^{\lceil s_1(p^e - 1) \rceil} \dots \bb_m^{\lceil s_m(p^e - 1) \rceil} \right)\\
\subseteq \sum_{\phi,\,  D_{\phi} \geq (p^e - 1)\Gamma + b(p^e -1) \Div(f)} \phi\left(F^e_* c  \ba^{\lceil t(p^e - 1) \rceil} \bb_1^{\lceil s_1(p^e - 1) \rceil} \dots \bb_m^{\lceil s_m(p^e - 1) \rceil} \right)\\
= \tau_b(R; \Delta, \ba^t \bb_1^{s_1} \dots \bb_m^{s_m}).
\end{split}
\]
and so $\tau_b(R; \Gamma, f^b \ba^t \bb_1^{s_1} \dots \bb_m^{s_m}) \subseteq \tau_b(R; \Delta, \ba^t \bb_1^{s_1} \dots \bb_m^{s_m})$.  For the converse inclusion, observe first that
\[
 \Div(f^{\lceil b(p^e - 1)\rceil}) - b(p^e - 1) \Div(f) \leq \Div(f).
\]
  Thus, since $cf$ is also a test element,
\[
 \begin{split}
  \tau_b(R, \Delta, \ba^t \bb_1^{s_1} \dots \bb_m^{s_m}) = \sum_{\phi,\, D_{\phi} \geq (p^e - 1)\Delta} \phi\left(F^e_* c f  \ba^{\lceil t(p^e - 1) \rceil} \bb_1^{\lceil s_1(p^e - 1) \rceil} \dots \bb_m^{\lceil s_m(p^e - 1) \rceil} \right)\\
 = \sum_{\phi,\, D_{\phi} \geq (p^e - 1)\Gamma + b(p^e -1) \Div(f)} \phi\left(F^e_* c f \ba^{\lceil t(p^e - 1) \rceil} \bb_1^{\lceil s_1(p^e - 1) \rceil} \dots \bb_m^{\lceil s_m(p^e - 1) \rceil} \right)\\
= \sum_{\phi,\, D_{\phi} \geq (p^e - 1)\Gamma + b(p^e -1) \Div(f) + \Div(f) } \phi\left(F^e_* c \ba^{\lceil t(p^e - 1) \rceil}  \bb_1^{\lceil s_1(p^e - 1) \rceil} \dots \bb_m^{\lceil s_m(p^e - 1) \rceil} \right)\\
\subseteq \sum_{\phi,\, D_{\phi} \geq (p^e - 1)\Gamma + \Div(f^{\lceil b(p^e - 1)\rceil}) } \phi\left(F^e_* c \ba^{\lceil t(p^e - 1) \rceil} \bb_1^{\lceil s_1(p^e - 1) \rceil} \dots \bb_m^{\lceil s_m(p^e - 1) \rceil}\right)\\
= \sum_{\phi,\, D_{\phi} \geq (p^e - 1)\Gamma } \phi\left(F^e_* c f^{\lceil b(p^e - 1)\rceil}\ba^{\lceil t(p^e - 1) \rceil} \bb_1^{\lceil s_1(p^e - 1) \rceil} \dots \bb_m^{\lceil s_m(p^e - 1) \rceil}\right) \\
= \tau_b(R, \Gamma, f^b \ba^t \bb_1^{s_1} \dots \bb_m^{s_m})
 \end{split}
\]
and so $\tau_b(R; \Delta, \ba^t \bb_1^{s_1} \dots \bb_m^{s_m}) \subseteq \tau_b(R; \Gamma, f^b \ba^t \bb_1^{s_1} \dots \bb_m^{s_m})$ as desired.
\end{proof}

We also need a very special case of Skoda's theorem.

\begin{lemma} \cite[Theorem 4.1]{HaraTakagiOnAGeneralizationOfTestIdeals}
\label{Lemma:SkodaTypeTheorem}
Suppose that $X = \Spec R$, $\Delta > 0$, $\ba \subseteq R$ and $t \geq 0$ is as above.  Further suppose that $f \in R$ is a non-zero element.  Then
\[
 \tau_b(X; \Delta + \Div(f), \ba^t \bb_1^{s_1} \dots \bb_m^{s_m}) = f \tau_b(X; \Delta, \ba^{t} \bb_1^{s_1} \dots \bb_m^{s_m}) .
\]
\end{lemma}
\begin{proof}
 We leave the proof to reader; see \cite[Theorem 4.2]{HaraTakagiOnAGeneralizationOfTestIdeals} and \cite[Lemma 3.26]{BlickleSchwedeTakagiZhang}.
\end{proof}

And now we can prove the following result.

\begin{theorem}
\label{Thm:TranslateTestIdealQuestion}
 Suppose that $R$ is an $F$-finite normal domain and further suppose that $(X = \Spec R, \Delta, \ba^t)$ is a triple where $K_X + \Delta$ is $\bR$-Cartier.  Then for each point $x \in X$, there exists an open set $U = \Spec R' = \Spec R[h^{-1}]$ containing $x \in X$ with the following properties:  There exists an effective $\bQ$-divisor $\Gamma$ on $U$, elements $f_{1}, \dots f_{m} \in R' \setminus \{0\}$ and nonnegative real numbers $b_1, \dots, b_{m}$ such that
\begin{itemize}
 \item[(1)]  $K_{U} + \Gamma$ is $\bQ$-Cartier with index not divisible by $p > 0$ and furthermore, $(p^e - 1)(K_{U} + \Gamma) \sim 0$ for some integer $e > 0$.
 \item[(2)]  The $F$-jumping numbers of $\tau_b(U, \Gamma, f_1^{b_1} \dots f_m^{b_m} (\ba R')^t)$ are the same as the $F$-jumping numbers of $\tau_b(U, \Delta|_{U}, (\ba R')^t)$ (both sets of jumping numbers are with respect to $t$).
\end{itemize}
\end{theorem}
\begin{proof}
Choose a non-zero section $\phi$ of $\Hom_R(F^e_* R, R)$ and set $\Gamma = {1 \over p^e - 1} D_{\phi}$, it follows that $K_X + \Gamma$ satisfies condition (1) on $X$.  Therefore, $(K_X + \Delta) - (K_X + \Gamma) = \Delta - \Gamma$ is $\bR$-Cartier and so we may write
$\Delta - \Gamma = d_1 D_1 + \dots + d_m D_m$ for some integral effective Cartier divisors $D_i$ and real numbers $d_i \in \bR$.  We choose our open set $U = \Spec R[h^{-1}] = \Spec R'$ to be any such set containing $x \in X$ where all of the $D_i|_{U}$ are principal divisors.

Now write $D_i|_{U} = \Div(f_i)$ for some $f_i \in R' \setminus 0$ and also by abuse of notation denote $\Gamma := \Gamma|_{U}$.  Choose natural numbers $l_i$ such that $b_i := l_i + d_i > 0$ for all $i$ and set $g := f_1^{l_1} \dots f_m^{l_m} \in R'$.  Notice that $(\Delta|_{U} + \Div(g)) - \Gamma = b_1 \Div(f_1) + \dots + b_m \Div(f_m).$

By Lemma \ref{Lemma:SkodaTypeTheorem}, the $F$-jumping numbers of $\tau_b(U; \Delta|_{U}, (\ba R')^t)$ and the $F$-jumping numbers of $\tau_b(U; \Delta|_{U} + \Div(g), (\ba R')^t)$ coincide.
Now using Lemma \ref{lemma:TwistIn} we have
\begin{align*}
 & \, \tau_b(U; \Delta|_{U} + \Div(g), (\ba R')^t) \\
 = & \,  \tau_b(U; \Gamma + b_1 \Div(f_1) + \dots + b_m \Div(f_m), (\ba R')^t) \\
 = & \, \tau_b(U; \Gamma, f_1^{b_1} \dots f_m^{b_m} (\ba R')^t)
\end{align*}
which proves the theorem.
\end{proof}

\begin{remark}
\label{remNeedOneOverP}
If, in Theorem \ref{Thm:TranslateTestIdealQuestion}, $K_X + \Delta$ is $\bQ$-Cartier, then one needs only a single $f_1^{b_1}$ (and no other $f_i^{b_i}$).  However, if the index of $K_X + \Delta$ is divisible by $p > 0$, then it follows by construction that $b_1$ will be a rational number with denominator divisible by $p > 0$.
\end{remark}

We are now in a position to prove the discreteness of the $F$-jumping numbers in the case that $X$ is essentially of finite type over a field.  The proof idea follows the usual lines.

\begin{theorem}\label{MainTheorem}
Suppose that $R$ is a normal domain essentially of finite type over an $F$-finite field.  Further suppose that $\ba \subseteq R$ is an ideal and $\Delta$ is an $\bR$-divisor on $X = \Spec R$ such that $K_X + \Delta$ is $\bR$-Cartier (for example, this holds if $\Delta = 0$ and $R$ is $\bQ$-Gorenstein).  Then, as $t$ varies, the $F$-jumping numbers of $\tau_b(R; \Delta, \ba^t)$ have no limit points -- they are discrete.
\end{theorem}
\begin{proof}
By \cite[Proposition 3.28]{BlickleSchwedeTakagiZhang}, it is sufficient to answer this question on a finite affine cover of $X$.  Therefore, we reduce to the case that $X$ is one of the charts from Theorem \ref{Thm:TranslateTestIdealQuestion}.  In particular, it is sufficient to prove our result for triples of the form $\tau_b(R; \Gamma, f_1^{b_1} \dots f_m^{b_m} \ba^t)$ where $(p^e-1)(K_X + \Gamma) \sim 0$ for some $e > 0$. Using \cite[Lemma 4.2, Proposition 3.28]{BlickleSchwedeTakagiZhang}, one can further assume that $R$ is of finite type over an $F$-finite field of characteristic $p > 0$.  One then has two options:
\begin{itemize}
\item[(a)]  Mimic the proof of the main result of \cite[Section 4]{BlickleSchwedeTakagiZhang}.  In other words, use the methods of $F$-adjunction (as worked out in \cite{SchwedeFAdjunction} and \cite{BlickleSchwedeTakagiZhang}) to reduce to the case where $R$ is a polynomial ring and then use degree bounding methods similar to those found in \cite{BlickleMustataSmithDiscretenessAndRationalityOfFThresholds}.  Note that in \cite{BlickleSchwedeTakagiZhang}, one worked with triples $(R, \Delta, \ba^t)$ and not with the more complicated objects $(R, \Gamma, f_1^{b_1} \dots f_m^{b_m} \ba^t)$, but the methods are easily generalized to our setting.
\item[(b)]  Use the new language of \cite[Section 4]{BlickleTestIdealsViaAlgebras}.  We claim that the algebra of $p^{-e}$-linear maps associated to the triple $(R, \Gamma, f_1^{b_1} \dots f_m^{b_m})$, as in \cite[Remark 3.10]{SchwedeTestIdealsInNonQGor}, is ``gauge bounded'' (see \cite[Definition 4.7]{BlickleTestIdealsViaAlgebras}).  To see this claim, note that by \cite[Lemma 3.9]{SchwedeFAdjunction} or \cite[Remark 4.4]{SchwedeTestIdealsInNonQGor}, the Cartier-algebra associated to $(R, \Gamma)$ is finitely generated and thus gauge bounded by \cite[Proposition 4.8]{BlickleTestIdealsViaAlgebras}.  It follows then from \cite[Proposition 4.13]{BlickleTestIdealsViaAlgebras} that the Cartier-algebra associated to $(R, \Gamma, f_1^{b_1} \dots f_m^{b_m})$ is also gauge bounded as claimed.  To finish the proof, apply \cite[Theorem 4.14]{BlickleTestIdealsViaAlgebras}.
\end{itemize}
In either case, the result follows easily from the theories previously developed.
\end{proof}

\section{On the question of rationality}

Note that the usual way to prove the rationality of the $F$-jumping numbers employs the following theorem. First recall that a pair $(X, \Delta)$ is called \emph{log $\bQ$-Gorenstein with index $n$} if $n(K_X + \Delta)$ is Cartier and $n > 0$ is the smallest integer with this property.

\begin{theorem}\cite{BlickleMustataSmithDiscretenessAndRationalityOfFThresholds}, \cite{BlickleSchwedeTakagiZhang}
 Suppose $(X, \Delta)$ is log $\bQ$-Gorenstein with index $n$ such that $n$ divides $(p^e - 1)$ for some fixed $e > 0$.  Further suppose that $\ba$ is an ideal sheaf of $X$.  Then if $t_0$ is an $F$-jumping number of $\tau(X; \Delta,  \ba^t)$, then $p^e t_0$ is also an $F$-jumping number.
\end{theorem}

However, without the ``index not divisible by $p$'' assumption, this theorem is false.  Consider the following example (which in some sense typical by Remark \ref{remNeedOneOverP}).

\begin{example}
Set $X = \bA^1_k = \Spec k[x]$, $\Delta = {1 \over p} \Div(x)$ and $\ba = (x)$.  Then the $F$-jumping number of $(X, \Delta, \ba^t) = (X, (x)^{1/p} \ba^t)$ with respect to $t$ are
\[
 {p - 1 \over p}, {2p - 1 \over p}, {3p - 1 \over p}, \ldots .
\]
In particular, $p$ (or $p^e)$) times any of them is not an $F$-jumping number.
\end{example}

\def\cprime{$'$} \def\cprime{$'$}
  \def\cfudot#1{\ifmmode\setbox7\hbox{$\accent"5E#1$}\else
  \setbox7\hbox{\accent"5E#1}\penalty 10000\relax\fi\raise 1\ht7
  \hbox{\raise.1ex\hbox to 1\wd7{\hss.\hss}}\penalty 10000 \hskip-1\wd7\penalty
  10000\box7}
\providecommand{\bysame}{\leavevmode\hbox to3em{\hrulefill}\thinspace}
\providecommand{\MR}{\relax\ifhmode\unskip\space\fi MR}
\providecommand{\MRhref}[2]{%
  \href{http://www.ams.org/mathscinet-getitem?mr=#1}{#2}
}
\providecommand{\href}[2]{#2}


\begin{thebibliography}{MTW05}

\bibitem[Bli09]{BlickleTestIdealsViaAlgebras}
{\sc M.~Blickle}: \emph{Test ideals via algebras of $p^{-e}$-liner maps},
  arXiv:0912.2255.

\bibitem[BMS08]{BlickleMustataSmithDiscretenessAndRationalityOfFThresholds}
{\sc M.~Blickle, M.~Musta{\c{t}}{\u{a}}, and K.~Smith}: \emph{Discreteness and
  rationality of {F}-thresholds}, Michigan Math. J. \textbf{57} (2008), 43--61.

\bibitem[BMS09]{BlickleMustataSmithFThresholdsOfHypersurfaces}
{\sc M.~Blickle, M.~Musta{\c{t}}{\u{a}}, and K.~E. Smith}:
  \emph{{$F$}-thresholds of hypersurfaces}, Trans. Amer. Math. Soc.
  \textbf{361} (2009), no.~12, 6549--6565. {\sf\scriptsize MR2538604}

\bibitem[BSTZ10]{BlickleSchwedeTakagiZhang}
{\sc M.~Blickle, K.~Schwede, S.~Takagi, and W.~Zhang}: \emph{Discreteness and
  rationality of {$F$}-jumping numbers on singular varieties}, Math. Ann.
  \textbf{347} (2010), no.~4, 917--949. {\sf\scriptsize 2658149}.
  
\bibitem[DH09]{DeFernexHaconSingsOnNormal}
{\sc T.~{De Fernex} and C.~Hacon}: \emph{Singularities on normal varieties},
  Compos. Math. \textbf{145} (2009), no.~2, 393--414.

\bibitem[ELSV04]{EinLazSmithVarJumpingCoeffs}
{\sc L.~Ein, R.~Lazarsfeld, K.~E. Smith, and D.~Varolin}: \emph{Jumping
  coefficients of multiplier ideals}, Duke Math. J. \textbf{123} (2004), no.~3,
  469--506. {\sf\scriptsize MR2068967 (2005k:14004)}

\bibitem[Har06]{HaraMonskyFPureThresholdsAndFJumpingExponents}
{\sc N.~Hara}: \emph{F-pure thresholds and {F}-jumping exponents in dimension
  two}, Math. Res. Lett. \textbf{13} (2006), no.~5-6, 747--760, With an
  appendix by Paul Monsky. {\sf\scriptsize MR2280772}

\bibitem[HT04]{HaraTakagiOnAGeneralizationOfTestIdeals}
{\sc N.~Hara and S.~Takagi}: \emph{On a generalization of test ideals}, Nagoya
  Math. J. \textbf{175} (2004), 59--74. {\sf\scriptsize MR2085311
  (2005g:13009)}

\bibitem[HY03]{HaraYoshidaGeneralizationOfTightClosure}
{\sc N.~Hara and K.-I. Yoshida}: \emph{A generalization of tight closure and
  multiplier ideals}, Trans. Amer. Math. Soc. \textbf{355} (2003), no.~8,
  3143--3174 (electronic). {\sf\scriptsize MR1974679 (2004i:13003)}

\bibitem[Hoc07]{HochsterFoundations}
{\sc M.~Hochster}: \emph{Foundations of tight closure theory}, lecture notes
  from a course taught on the University of Michigan Fall 2007 (2007).

\bibitem[HH90]{HochsterHunekeTC1}
{\sc M.~Hochster and C.~Huneke}: \emph{Tight closure, invariant theory, and the
  {B}rian\c con-{S}koda theorem}, J. Amer. Math. Soc. \textbf{3} (1990), no.~1,
  31--116. {\sf\scriptsize MR1017784 (91g:13010)}

\bibitem[KLZ09]{KatzmanLyubeznikZhangOnDiscretenessAndRationality}
{\sc M.~Katzman, G.~Lyubeznik, and W.~Zhang}: \emph{On the discreteness and
  rationality of {$F$}-jumping coefficients}, J. Algebra \textbf{322} (2009),
  no.~9, 3238--3247. {\sf\scriptsize MR2567418}

\bibitem[MTW05]{MustataTakagiWatanabeFThresholdsAndBernsteinSato}
{\sc M.~Musta{\c{t}}{\v{a}}, S.~Takagi, and K.-i. Watanabe}: \emph{F-thresholds
  and {B}ernstein-{S}ato polynomials}, European Congress of Mathematics, Eur.
  Math. Soc., Z\"urich, 2005, pp.~341--364. {\sf\scriptsize MR2185754
  (2007b:13010)}

\bibitem[MY09]{MustataYoshidaTestIdealVsMultiplierIdeals}
{\sc M.~Musta{\c{t}}{\u{a}} and K.-I. Yoshida}: \emph{Test ideals vs.
  multiplier ideals}, Nagoya Math. J. \textbf{193} (2009), 111--128.
  {\sf\scriptsize MR2502910}

\bibitem[Sch09a]{SchwedeFAdjunction}
{\sc K.~Schwede}: \emph{{$F$}-adjunction}, Algebra Number Theory \textbf{3}
  (2009), no.~8, 907--950.

\bibitem[Sch09b]{SchwedeTestIdealsInNonQGor}
{\sc K.~Schwede}: \emph{Test ideals in non-{Q}-{G}orenstein rings},
  arXiv:0906.4313, to appear in Trans. Amer. Math. Soc.

\bibitem[Sch10]{SchwedeCentersOfFPurity}
{\sc K.~Schwede}: \emph{Centers of {$F$}-purity}, Math. Z. \textbf{265} (2010),
  no.~3, 687--714.

\bibitem[Tak04]{TakagiInterpretationOfMultiplierIdeals}
{\sc S.~Takagi}: \emph{An interpretation of multiplier ideals via tight
  closure}, J. Algebraic Geom. \textbf{13} (2004), no.~2, 393--415.
  {\sf\scriptsize MR2047704 (2005c:13002)}

\bibitem[Tak06]{TakagiFormulasForMultiplierIdeals}
{\sc S.~Takagi}: \emph{Formulas for multiplier ideals on singular varieties},
  Amer. J. Math. \textbf{128} (2006), no.~6, 1345--1362. {\sf\scriptsize
  MR2275023 (2007i:14006)}

\bibitem[Urb10]{UrbinatiDiscOfNonQGorenVars}
{\sc S.~Urbinati}: \emph{Discrepancies of non-${\bQ}$-{G}orenstein varieties},
  arXiv:1001.2930.

\end{thebibliography}
\end{document}